\documentclass[runningheads]{llncs}
\usepackage{graphicx}

\usepackage[utf8x]{inputenc}
\usepackage[boxruled,linesnumbered]{algorithm2e}
\usepackage{subfigure}
\usepackage{amsmath}
\usepackage{amssymb}
\usepackage[colorlinks]{hyperref}
\usepackage{xspace}
\usepackage{xcolor}
\usepackage[small]{complexity}
\usepackage{pbox}
\usepackage{cite}
\usepackage{version}

\excludeversion{omit}

\usepackage{hyperref}
\hypersetup{colorlinks = true}
\usepackage[capitalize,nameinlink,noabbrev]{cleveref}

\usepackage{tikz}

\usetikzlibrary{mindmap,positioning}
\usetikzlibrary{graphs}
\usetikzlibrary[graphs]
\usetikzlibrary{patterns}
\usetikzlibrary{arrows,automata,positioning}

\usetikzlibrary{decorations.pathreplacing}
\usetikzlibrary{decorations.pathmorphing, patterns,shapes}

\pgfdeclarelayer{bg}    
\pgfsetlayers{bg,main}  

\tikzset{faded/.style={gray,very thin}}
\tikzset{vertex/.style={draw,circle,minimum size=20pt,inner sep=0pt}}
\tikzset{novertex/.style={circle,minimum size=10pt,inner sep=0pt}}
\tikzset{point/.style={circle,minimum size=0pt,inner sep=0pt}}
\tikzset{blackvertex/.style={draw,circle,minimum size=10pt,inner sep=0pt, fill=black}}
\tikzset{smallbv/.style={draw,circle,minimum size=6pt,inner sep=0pt, fill=black}}
\tikzset{grayvertex/.style={draw,circle,minimum size=10pt,inner sep=0pt, fill=gray}}
\tikzset{redvertex/.style={draw,circle,minimum size=10pt,inner sep=0pt, fill=red}}
\tikzset{redvertexfaded/.style={draw,circle,faded,minimum size=10pt,inner sep=0pt, fill=red!50}}
\tikzset{greenvertex/.style={draw,circle,minimum size=10pt,inner sep=0pt, fill=green}}
\tikzset{greenvertexfaded/.style={draw,circle,faded,minimum size=10pt,inner sep=0pt, fill=green!50}}
\tikzset{bluevertex/.style={draw,circle,minimum size=10pt,inner sep=0pt, fill=blue}}
\tikzset{bluevertexfaded/.style={draw,circle,faded,minimum size=10pt,inner sep=0pt, fill=blue!50}}
\tikzset{yellowvertex/.style={draw,circle,minimum size=10pt,inner sep=0pt, fill=yellow}}
\tikzset{yellowvertexfaded/.style={draw,circle,faded,minimum size=10pt,inner sep=0pt, fill=yellow!50}}

\tikzset{edge/.style = {->,> = latex',very thick}}
\tikzset{tedge/.style = {->,> = latex',thick}}
\tikzset{boldedge/.style = {->,> = latex',line width=1.5pt}}
\tikzset{rededge/.style = {->,> = latex',ultra thick, draw=red}}
\tikzset{uedge/.style = {-,> = latex',very thick}}

\tikzset{
    invisible/.style={opacity=0,text opacity=0},
    visible on/.style={alt=#1{}{invisible}},
    alt/.code args={<#1>#2#3}{%
      \alt<#1>{\pgfkeysalso{#2}}{\pgfkeysalso{#3}} 
    },
}

\newcommand{\julio}[1]{\textcolor{red}{Julio: #1}}

\DeclareMathOperator{\hn}{hn}
\DeclareMathOperator{\gn}{gn}
\DeclareMathOperator{\itn}{in}

\DeclareMathOperator{\gc}{g}
\DeclareMathOperator{\pc}{P_3}
\DeclareMathOperator{\psc}{P_3^*}
\DeclareMathOperator{\ogc}{\overrightarrow{g}}
\DeclareMathOperator{\opc}{\overrightarrow{P_3}}
\DeclareMathOperator{\opsc}{\overrightarrow{P_3^*}}
\DeclareMathOperator{\xc}{{\cal X}}
\DeclareMathOperator{\oxc}{\overrightarrow{\xc}}
\DeclareMathOperator{\arc}{arc}

\DeclareMathOperator{\ohn}{\overrightarrow{\hn}}
\DeclareMathOperator{\ogcn}{\overrightarrow{\gn}}
\DeclareMathOperator{\oin}{\overrightarrow{\itn}}

\DeclareMathOperator{\ing}{\itn_{\gc}}
\DeclareMathOperator{\inp}{\itn_{\pc}}
\DeclareMathOperator{\inps}{\itn_{\psc}}
\DeclareMathOperator{\inx}{\itn_{\xc}}

\DeclareMathOperator{\hng}{\hn_{\gc}}
\DeclareMathOperator{\hnp}{\hn_{\pc}}
\DeclareMathOperator{\hnps}{\hn_{\psc}}
\DeclareMathOperator{\hnx}{\hn_{\xc}}

\DeclareMathOperator{\oinx}{\oin_{\xc}}
\DeclareMathOperator{\ohnx}{\ohn_{\xc}}

\DeclareMathOperator{\oing}{\oin_{\gc}}
\DeclareMathOperator{\ohng}{\ohn_{\gc}}

\DeclareMathOperator{\oinp}{\oin_{\pc}}
\DeclareMathOperator{\ohnp}{\ohn_{\pc}}

\DeclareMathOperator{\oinps}{\oin_{\psc}}
\DeclareMathOperator{\ohnps}{\ohn_{\psc}}

\DeclareMathOperator{\phn}{hn^+}
\DeclareMathOperator{\pgn}{gn^+}
\DeclareMathOperator{\mhn}{hn^-}
\DeclareMathOperator{\mgn}{gn^-}

\DeclareMathOperator{\intfunc}{I}
\DeclareMathOperator{\ifg}{\intfunc_{\gc}}
\DeclareMathOperator{\ifp}{\intfunc_{\pc}}
\DeclareMathOperator{\ifps}{\intfunc_{\psc}}
\DeclareMathOperator{\ifxc}{\intfunc_{\xc}}

\DeclareMathOperator{\oi}{\overrightarrow{\intfunc}}
\DeclareMathOperator{\oifg}{\oi_{\gc}}
\DeclareMathOperator{\oifp}{\oi_{\pc}}
\DeclareMathOperator{\oifps}{\oi_{\psc}}

\DeclareMathOperator{\ext}{Ext}

\DeclareMathOperator{\hull}{H}
\DeclareMathOperator{\hullg}{\ensuremath{\hull_{\gc}}}
\DeclareMathOperator{\hullp}{\ensuremath{\hull_{\pc}}}
\DeclareMathOperator{\hullps}{\ensuremath{\hull_{\psc}}}
\DeclareMathOperator{\hullx}{\ensuremath{\hull_{\xc}}}

\DeclareMathOperator{\ohull}{\overrightarrow{\hull}}
\DeclareMathOperator{\ohullx}{\ohull_{\xc}}
\DeclareMathOperator{\ohullg}{\ohull_{\gc}}
\DeclareMathOperator{\ohullp}{\ohull_{\pc}}
\DeclareMathOperator{\ohullps}{\ohull_{\psc}}

\newtheorem{thm}{Theorem}
\newtheorem{myclaim}[thm]{Claim}
\newtheorem{prop}[thm]{Proposition}
\newtheorem{lem}[thm]{Lemma}
\newtheorem{cor}[thm]{Corollary}
\newtheorem{prob}[thm]{Problem}

\renewenvironment{proof}{\paragraph{\textbf{Proof:}}}{\hfill$\qed$}
\newenvironment{claimproof}[1]{\par\noindent\underline{Proof:}\space#1}{\hfill $\blacksquare$}
\newenvironment{sketch}{\paragraph{\textcolor{red}{Sketch:}}}{\hfill$\square$}

\usepackage{tabularx}
\usepackage{nameref}
\usepackage[skins,listings]{tcolorbox}
\newcommand{\pname}[1]{\textsc{#1}}
\newtcolorbox{ProblemBox}[2][]{
	colframe=black!20!white,
	coltitle=black,
	arc=1.5mm,
	boxrule=.15mm,
	colbacktitle=white,
	colback=white,
	enhanced,
	adjusted title=flush left,
	attach boxed title to top left={yshift=-3mm,xshift=3mm},
	boxed title style={colframe=white,righttitle=-3mm,lefttitle=-3mm},
	title=#2,#1
}
\newcommand{\problemDef}[4]{
	\begin{ProblemBox}[label={#4},nameref={#1}]{\pname{#1}}
		\begin{tabularx}{\textwidth}{l l} 
			\textnormal{Input:} & \quad \textnormal{#2} \\ 
			\textnormal{Question:} & \quad \textnormal{#3}
		\end{tabularx}
	\end{ProblemBox}
}

\newcommand{\paraProblemDef}[5]{
	\begin{ProblemBox}[label={#5},nameref={#1}]{\pname{#1}}
		\begin{tabularx}{\textwidth}{l l} 
			\textnormal{Input:} & \quad \textnormal{#2} \\ 
			\textnormal{Parameter:} & \quad \textnormal{#3}. \\ 
			\textnormal{Question:} & \quad \textnormal{#4}
		\end{tabularx}
	\end{ProblemBox}
}

\begin{document}
\title{On the hull and interval numbers of oriented graphs\thanks{Partly supported by FUNCAP-Pronem 4543945/2016, Funcap 186-155.01.00/21, CNPq 313153/2021-3.}}
%
%
\author{J.~Araújo\inst{1}
\and
A.~K.~Maia\inst{2}
\and
P.~P.~Medeiros\inst{1}
\and
L.~Penso\inst{3}
} 
\authorrunning{Araujo et al.}
%
\institute{Departamento de Matemática, 
Universidade Federal do Ceará, Brazil\\
\email{julio@mat.ufc.br, pedropmed1@gmail.com}         
\and
Departamento de Computação, 
Universidade Federal do Ceará, Brazil\\
\email{karolmaia@ufc.br}\\
\and
Universit\"at Ulm, Ulm, Germany\\
\email{lucia.penso@uni-ulm.de}}
\maketitle              
\begin{abstract}
            In this work, for a given oriented graph $D$, we study its interval and hull numbers, denoted by $\oin(D)$ and $\ohn(D)$, respectively, in the oriented geodetic, $\opc$ and $\opsc$ convexities. This last one, we believe to be formally defined and first studied in this paper, although its undirected version is well-known in the literature.
            
            Concerning bounds, for a strongly oriented graph $D$, and the oriented geodetic convexity, we prove that $\ohng(D)\leq m(D)-n(D)+2$ and that there is at least one such that $\ohng(D) = m(D)-n(D)$. We also determine exact values for the hull numbers in these three convexities for tournaments, which imply polynomial-time algorithms to compute them. These results allow us to deduce polynomial-time algorithms to compute $\ohnp(D)$ when the underlying graph of $D$ is split or cobipartite.
            
            Moreover, we provide a meta-theorem by proving that if deciding whether $\oing(D)\leq k$ or $\ohng(D)\leq k$ is \NP-hard or \W[i]-hard parameterized by $k$, for some $i\in\mathbb{Z_+^*}$, then the same holds even if the underlying graph of $D$ is bipartite. Next, we prove that deciding whether $\ohnp(D)\leq k$ or $\ohnps(D)\leq k$ is $\W[2]$-hard parameterized by $k$, even if $D$ is acyclic and its underlying graph is bipartite; that deciding whether $\ohng(D)\leq k$ is $\W[2]$-hard parameterized by $k$, even if $D$ is acyclic; that deciding whether $\oinp(D)\leq k$ or $\oinps(D)\leq k$ is $\NP$-complete, even if $D$ has no directed cycles and the underlying graph of $D$ is a chordal bipartite graph; and that deciding whether $\oinp(D)\leq k$ or $\oinps(D)\leq k$ is $\W[2]$-hard parameterized by $k$, even if the underlying graph of $D$ is split.
            
           Finally, also argue that the interval and hull numbers in the $\opc$ and $\opsc$ convexities can be computed in cubic time for graphs of bounded clique-width by using Courcelle's theorem.

\end{abstract}
\section{Introduction}
A \emph{convexity space} is an ordered pair $(V,\mathcal{C})$, where $V$ is an arbitrary set and $\mathcal{C}$ is a family of subsets of $V$, called \emph{convex sets}, that satisfies:
\begin{itemize}
  \item {\bf (C1)} $\emptyset,V\in \mathcal{C}$;
  \item {\bf (C2)} For all $\mathcal{C'}\subseteq \mathcal{C}$, we have $\bigcap \mathcal{C'}\in \mathcal{C}$;
  \item {\bf (C3)} Every nested union of convex sets is a convex set.
\end{itemize}


Convexity spaces are a classic topic, studied in different branches of mathematics.
The research of convexities applied to graphs is a more recent topic, from about 50 years ago. Classical convexity definitions and results motivated the definitions of some graph parameters (Carathéodory's number~\cite{barbosa-2012}, Helly's number~\cite{DOURADO2017134}, Radon's number~\cite{Delire1984}, hull number~\cite{THNOAG}, rank~\cite{Kante+2017}, etc.).
The complexity aspects related to the determination of the values of those parameters has been the central issue of several recent works in this field.
For basic notions on Graph Theory and Computational Complexity, the reader is referred to~\cite{BM2008,GJ90,BG2008}.

For graphs, convexity spaces are defined over the vertex set of a given (oriented) graph $G$. Here, we consider only finite, simple and non-null graphs. Thus, we always consider the set of vertices $V(G)$ to be finite and non-empty and the set of edges to be finite. This makes Condition (C3) irrelevant, since any such union will have a larger set.  

To define a convexity for a graph $G$, it remains to define the family $\mathcal{C}$. All graph convexities studied here are \emph{interval convexities}, i.e. defined by an \emph{interval function}. Given a graph $G = (V,E)$, an interval function $I: V\times V\to 2^V$ is a function such that $\{u,v\}\subseteq I(u,v)$ and $I(u,v)=I(v,u)$, for every $u,v\in V(G)$. If $S\subseteq V\times V$, with a slight abuse of notation, we define $I(S) = \bigcup_{(u,v)\in S} I(u,v)$.  We say that $C\subseteq V$ is \emph{convex} (or \emph{$I$-closed}) if $I(C\times C) = C$.
Let the family $\mathcal{C}$ be composed by all $I$-closed subsets of $V$. Thus, $(V,\mathcal{C})$ is a convexity space~\cite{CSIG,Calder1971}. Note that the unitary subsets of $V$, as well as the set $V$ itself, are convex. Such sets will be called \emph{trivial} convex sets.

Given an undirected graph $G = (V,E)$, when $I(u,v)$ returns the set of all vertices belonging to some shortest $u,v$-path (a.k.a. $u,v$-geodesic), we have the \emph{geodetic convexity}~\cite{FJ.86}. If $I(u,v)$ returns the set of all vertices that lie in some path on three vertices having $u$ and $v$ as endpoints, then we have the $P_3$-convexity~\cite{PWW.08}. In case $I(u,v)$ returns the set vertices that lie in some shortest $(u,v)$-path of length two, then we have the $P_3^*$-convexity~\cite{ASSS.18}. The interval functions in the geodetic, $P_3$ and $P_3^*$ convexities are denoted here by $\ifg(G)$, $\ifp(G)$ and $\ifps(G)$, respectively. Some other graph convexities have been studied in the literature~\cite{D.88}.

If $(V,\mathcal{C})$ is a convexity space, the \emph{convex hull} of $S\subseteq V$ with respect to $(V,\mathcal{C})$ is the unique inclusion-wise minimum $C\in \mathcal{C}$ containing $S$ and we denoted it by $\hull_{\mathcal{C}}(S)$. 
In the case of an interval convexity $(V,\mathcal{C})$, $\hull_{\mathcal{C}}(S)$ can be obtained by the interval function $I$ that defines $\mathcal{C}$, for any $S\subseteq V$, as follows: 
$$I^k(S) = \begin{cases}S, \text{ if }k=0;\\I(I^{k-1}(S)),\text{ otherwise.}\end{cases}$$

Since we deal with interval convexities defined over a finite graph $G$, there is an integer $0\leq k\leq n(G)$ such that $\hull_{\mathcal{C}}(S) = \bigcup_{i=0}^k I^i(S)$ (see~\cite{CSIG,Calder1971} for details and further references). 
The convex hulls of a subset $S\subseteq V$ in the geodetic, $P_3$ and $P_3^*$ convexities are denoted by $\hullg(S)$, $\hullp(S)$ and $\hullps(S)$.  

For each graph convexity, as we have mentioned before, several parameters have been defined in the literature. We focus in two of them. Given a graph convexity defined by an interval function $I$ over the vertex set of a graph $G=(V,E)$, an \emph{interval set} (resp. \emph{hull set}) $S\subseteq V$ satisfies that $I(S) = V$ (resp. $\hull(S) = V$). The \emph{interval number} (resp. \emph{hull number}) is the cardinality of a smallest interval set (resp. hull set) $S\subseteq V$ of $G$. The interval numbers in the geodetic, $P_3$ and $P_3^*$ convexities are denoted by $\ing(G)$, $\inp(G)$ and $\inps(G)$, respectively. Analogously, the hull number in these convexities are denoted by $\hng(G)$, $\hnp(G)$ and $\hnps(G)$. We emphasize that an interval set in the geodetic convexity is also known as \emph{geodetic set} and the interval number in this convexity is also called \emph{geodetic number} and it is often denoted by $\gn(G) = \ing(G)$.

\subsection{Convexity on oriented graphs}

Although graph convexities and their different parameters have been extensively studied
in recent years, this topic has not yet been equally investigated considering
directed or oriented graphs.

An oriented graph $D$ is an orientation of a simple graph $G$, i.e., $D$ is obtained from $G$ by choosing an orientation ($u$ to $v$ or $v$ to $u$) for each edge $uv$ of $G$.

For a given oriented graph $D$, the following convexities have been studied in the literature. 

    The \emph{geodetic convexity}, that we also refer as $\ogc$-convexity, over an oriented graph $D$ is defined by the interval function $\oifg$ that returns, for a pair $(u,v)\in V\times V$, the set of vertices that lie in any {\bf shortest directed} $(u,v)$-path or in any {\bf shortest directed} $(v,u)$-path.

    The \emph{two-path convexity}, which we will more often call \emph{$\opc$-convexity}, over an oriented graph $D$ is defined by the interval function $\oifp$ that returns, for a pair $(u,v)\in V\times V$ the set of vertices that lie in any {\bf directed} $(u,v)$-path of length two or in any {\bf directed} $(v,u)$-path of length two.
    
    The \emph{distance-two convexity}, to which we refer as \emph{$\opsc$-convexity}, over an oriented graph $D$ is defined by the interval function $\oifps$ that returns, for a pair $(u,v)\in V\times V$ the set of vertices that lie in any {\bf shortest directed} $(u,v)$-path of length two or in any {\bf shortest directed} $(v,u)$-path of length two. Note that thus there is no arc $(u,v)$ in the first case, nor $(v,u)$ in the second.

    Now, let consider $\xc\in\{\gc,\pc,\psc\}$. A convex set $S$ in $\oxc$-convexity will be denoted as \emph{$\oxc$-convex}. 
    The convex hull of a set $S\subseteq V(D)$ in the $\oxc$-convexity is denoted by $\ohullx(S)$. An interval (resp. hull) set in the $\oxc$-convexity will be referred as an \emph{interval} (resp. \emph{hull}) \emph{$\oxc$-set}. The interval and hull numbers in the $\oxc$-convexity of an oriented graph $D$ are denoted here by $\overrightarrow{in}_{\xc}(D)$ and $\overrightarrow{hn}_{\xc}(D)$, respectively.

    It is important to emphasize that as $D$ is an orientation of a simple graph, then it cannot have both arcs $(u,v)$ and $(v,u)$, for distinct $u,v\in V(D)$.
    Thus, the parameters in the oriented versions are not equivalent to the undirected ones.


%



    Although the first papers related to convexity in graphs study directed graphs \cite{CIT,EITIST,SROST}, most of the papers we can find in the literature about graph convexities deal with undirected graphs. 
    For instance, the hull and geodetic numbers with respect to undirected graphs~\cite{THNOAG,TGNOAG} were first studied in the literature around a decade before their corresponding directed versions~\cite{THNOAOG,TGNOAOG}.
    
    With respect to the directed case, most results in the literature provide bounds on the maximum and minimum values of $\ohn(D(G))$ and $\ogcn(D(G))$ among all possible orientations $D(G)$ of a given undirected simple graph $G$~\cite{THNOAOG,TGNOAOG,OCGAHNIG}.
    It is important to emphasize the results on the parameter $\phn(G)$, the upper orientable hull number of a simple graph \( G \), since these are the only ones related to the upper bound we present.
    Such parameter is defined in~\cite{THNOAOG}, as the maximum value of $\ohn(D)$ among all possible orientations $D$ of a simple graph $G$, and it is easy to see that \( \phn(G) = n(G) \) if and only if there is an orientation \( D \) of \( G \) such that every vertex $V(D)\setminus \{v\}$ is convex i.e. $v\in V(D)$ is extreme.
    They also compare this parameter with others, such as the lower orientable hull number (\( \mhn(G) \)) and the lower and upper orientable geodetic numbers (\( \mgn(G) \) and \( \pgn(G) \) respectively), which are defined analogously.
    
    There are also few results about some related parameters: the forcing hull and geodetic numbers~\cite{TFHAFGNOG,TFGNOAG}, the pre-hull number~\cite{OTGPHNOAG} and the Steiner number~\cite{OTSGAHNOG} are a few examples.
    
    In~\cite{PWW.08}, the authors provide upper bounds for the parameters rank, Helly number, Radon number, Carathéodory number, and hull number when restricted to multipartite tournaments.
    
    In~\cite{AA2021}, the authors focus in the hull and interval numbers for the geodetic convexity in oriented graphs. They first present a general upper bound to $\ohng(D)$, for any oriented graph $D$, and present a tight upper bound for a tournament $T$ satisfying $\ohng(T) = \frac{2}{3}n(T)$. Then, they use the result in the undirected case proved in~\cite{CIPCTHN} to show that computing the geodetic hull number of an oriented graph $D$ is \NP-hard even if the underlying graph of $D$ is a partial cube. Although they do not emphasize, their reduction builds a strongly oriented graph $D$. They also show that deciding whether $\oing(D)\leq k$ for an oriented graph $D$ whose underlying graph is bipartite or cobipartite or split is a \W[2]-hard problem when parameterized by $k$. 
    
    Once more, although the authors do not emphasize, their reduction to the class of bipartite graphs also works for the $\opc$ and $\opsc$ convexities as it uses only paths on three vertices, much are always shortest due to the bipartition, and so does their reduction to cobipartite graphs for the $\opsc$ convexity. Their reduction to split graphs uses paths on four vertices, and their reduction to cobipartite graphs does not hold for the $\opc$ as it is has many paths on three vertices that are not shortest and could be used in the $\opc$.
    
    Consequently, by the results in~\cite{AA2021}, determining whether $\oinps(D)\leq k$ is $\W[2]$-hard parameterized by $k$ even if the underlying graph $G$ of $D$ is bipartite or cobipartite, and determining whether $\oinp(D)\leq k$ is $\W[2]$-hard parameterized by $k$ even if the underlying graph $G$ of $D$ is bipartite.
    They also present polynomial-time algorithms to compute both parameters for any oriented cactus graph and they present a 
    tight upper bound for the geodetic hull number of an oriented split graph.
    
    \subsection{Our contributions and organization of the paper}
    
    We first present some basic results concerning the elements of hull and interval sets at the geodetic, $\opc$ and $\opsc$ convexities, in Section~\ref{sec:preliminaries}. 
    
    Then dedicate Section~\ref{sec:geodetic} to the results concerning bounds and equalities for the oriented hull number. In Section~\ref{subsec:strongly}, for a strongly oriented graph $D$, we prove that $\ohng(D)\leq m(D)-n(D)+2$ and that there is a strongly oriented graph such that $\ohng(D) = m(D)-n(D)$. In Section~\ref{subsec:tournaments}, we determine exact values for the hull numbers in the geodetic and $\opsc$ convexities for tournaments. These formulas imply polynomial-time algorithms for computing these parameters for this particular case of tournaments. The case of $\opc$ was already studied~\cite{HW1996}. We present a shorter proof that $\ohnp(T)\leq 2$ for a given tournament $T$, and we use this result to deduce that $\ohnp(D)$ can be computed in polynomial time, whenever the underlying graph of $D$ is cobipartite or split.
    
    Section~\ref{sec:hardness} is devoted to hardness results from the computational complexity point of view.
    Section~\ref{subset:meta-theorem}, we provide a meta-theorem by proving that if deciding whether $\oing(D)\leq k$ or $\ohng(D)\leq k$ is \NP-hard or \W[i]-hard parameterized by $k$, for some $i\in\mathbb{Z_+^*}$, then the same holds even if the underlying graph of $D$ is bipartite.
    In the Sections~\ref{subsec:hardness-hull} and~\ref{subsec:hardness-interval}, we seek to complement the study of~\cite{AA2021} from the computational complexity point of view when restricted to bipartite, split, or cobipartite graphs. We prove that deciding whether $\ohnp(D)\leq k$ is $\W[2]$-hard parameterized by $k$ even if $D$ is acyclic and the underlying graph of $D$ is bipartite. With a slight modification in the reduction, we deduce that deciding whether $\ohng(D)\leq k$ is $\W[2]$-hard parameterized by $k$ even if $D$ is acyclic. Then, we prove that deciding whether $\oinp(D)\leq k$ is $\NP$-complete, even if $D$ has no directed cycles and the underlying graph of $D$ is a chordal bipartite graph. Since both results about the hardness of computing $\ohnp(D)$ hold for oriented graphs $D$ whose underlying graph is also bipartite, then both results also hold for $\ohnps(D)$ and $\oinps(D)$, respectively. Finally, we also prove that deciding whether $\oinp(D)\leq k$ or $\oinps(D)\leq k$ is $\W$[2]-hard when parameterized by $k$, even if the underlying graph of $D$ is split. 
            
    Since such parameters are hard to compute for bipartite graphs, a natural question is to ask whether they can be computed in polynomial time for trees. We argue in Section~\ref{subsec:coucelle} that interval and hull numbers in the $\opc$ and $\opsc$ convexities can be computed in polynomial time for graphs of bounded clique-width by using Courcelle's theorem. 
    
    \begin{table}[h!]
        \centering
        \begin{tabular}{c|c|c}
            Graph Class & Parameters & Complexity \\ \hline
            Tournaments & $\ohng$,$\ohnps$, 
 ($\ohnp$\cite{HW1996})
            & $\P$

            \\ \hline
            Underl. bipartite & ($\ohng \leq k$\cite{AA2021}) & $\NP$-complete \\
            & ($\oing, \oinp, \oinps \leq k$\cite{AA2021}) & $\W[2]$-hard by $k$ \\ 
            & $\ohnp,\ohnps \leq k$ & $\W[2]$-hard by $k$ \\ \hline
            Underl. chordal bipartite & $\oinp,\oinps\leq k$ & $\NP$-complete \\ \hline
            Underl. split & ($\oing\leq k$\cite{AA2021}) & $\W[2]$-hard by $k$  \\ 
            & $\oinp, \oinps\leq k$ & $\W[2]$-hard by $k$ \\
            & $\ohnp$ & \P \\
            \hline
            Underl. cobipartite & ($\oing, \oinps\leq k$~\cite{AA2021})& $\W[2]$-hard by $k$\\
            & $\ohnp$ & \P \\ \hline
            {\small General graphs} & {\tiny $\oinp,\oinps,\ohnp,\ohnps$}& {\footnotesize $FPT$ by clique width.}
        \end{tabular}
        \caption{Summary of complexity results according to main studied graph classes. In parentheses we have results from other authors.}
        \label{tab:complexity}
    \end{table}
    
    Finally, Section~\ref{sec:conclusion} presents avenues for further research. A summary of results is shown in Table \ref{tab:complexity}.

    
    \begin{omit}
    In Section~\ref{sec:further}, we present some directions for further research.
    \end{omit}

\section{Preliminaries}
\label{sec:preliminaries}
    
    Let $D$ be oriented graph. A vertex $v\in V(D)$ is \emph{transitive} if for any $u,w\in V(D)$ such that there are $(u,v),(v,w)\in A(D)$, then $(u,w)\in A(D)$ as well. A vertex $v\in V(D)$ is \emph{extreme} if $v$ is a source, or $v$ is a sink, or $v$ is transitive. For an oriented graph $D$, we denote by $\ext(D)$ the set of extreme vertices of $D$. 
    
    Consider $\xc\in\{\gc,\pc,\psc\}$. For a convex set $S$ in $\oxc$-convexity, $V(D)\setminus S$ is said to be \emph{$\oxc$-coconvex}. We make the following remark.

    \begin{prop}
    \label{prop:coconvex_sets}
    Let $D$ be an oriented graph and let $\xc \in \{\ogc,\opc, \opsc\}$. If $S\neq \emptyset$ is $\xc$-coconvex, then any interval $\xc$-set and any hull $\xc$-set contains at least one vertex of $S$. In particular, if $v\in D$ is a source or a sink, then $v$ necessarily belongs to any interval $\xc$-set and any hull $\xc$-set of $D$. Moreover, for $\xc\in\{\ogc,\opsc\}$, if $v\in D$ is transitive, then $v$ belongs to any interval $\xc$-set and any hull $\xc$-set of $D$.\end{prop}
    
    \begin{proof}
        If $S$ is $\xc$-coconvex, by definition one of its vertices should belong to any interval or hull  $\xc$-set of $D$.
        If $\{v\}$ is a source or a sink, no directed $(u,w)$-path, for some $u,w\in V(D)\setminus\{v\}$, has $v$ as internal vertex. Analogously, if $\{v\}$ is transitive, no shortest directed $(u,w)$-path has $v$ as internal vertex. 
    \end{proof}\\

    
    
    For an oriented graph $D$ and a vertex $v\in V(D)$, define the \emph{in-section} (resp. \emph{out-section}) of $v$ in $D$ as the set of vertices $u$ having a directed $(u,v)$-path (resp. $(v,u)$-path) in $D$, and denote it by $S_D^-(v)$ (resp. $S_D^+(v)$). In particular, if $u\in S_D^+(v)$, we say that $u$ is \emph{reachable} from $v$.

    We say that $D$ is a \emph{strongly oriented graph} if $D$ is strongly connected, that is, for every pair $u,v$ of distinct vertices of $D$, there is a path from $u$ to $v$ and a path from $v$ to $u$ in $D$. The maximal strongly connected subgraphs of a digraph $D$ are its \emph{strong components}. Every oriented graph can be decomposed in its strong components, and there is an acyclic ordering of them. 
    A minor remark is that:

    \begin{prop}
     \label{prop:inoutsections}
     If $D$ is a connected, but not strongly connected, oriented graph and $u\in V(D)$ does not lie in a sink (resp. source) strong component $C$ of $D$, then there exists a sink (resp. source) strong component $C'\subseteq D$ such that $C'\subseteq S_D^+(v)$ (resp. $C'\subseteq S_D^-(v)$) and $C'\neq C$.
     \end{prop}
     \begin{proof}
         If suffices to take a maximal directed path $P$ in $D$ starting at $u$ (resp. ending at $u$). The other endpoint of $P$ is either a sink (resp. source), or lies in a sink (resp. source) strong component $D'$ as required.
     \end{proof}

    For $\xc\in\{\gc,\psc\}$, 
    an $\oxc$-coconvex set $S\subseteq V(D)$ in an oriented graph $D$ satisfies the property that no shortest directed $(u,v)$-path (of length two, in case $\oxc=\opsc$) contains a vertex of $S$, for some $u,v\notin S$. This occurs in three situations: 
    \begin{itemize}
        \item there is no arc $(u,w)\in A(D)$ such that $u\in V(D)\setminus S$ and $w\in S$; in this case we say that $S$ is a \emph{source $\oxc$-coconvex set};
        \item there is no arc $(w,u)\in A(D)$ such that $u\in V(D)\setminus S$ and $w\in S$; in this case we say that $S$ is a \emph{sink $\oxc$-coconvex set};
        \item there are arcs $(u,w),(z,v)\in A(D)$ such that $u,v\in V(D)\setminus S$, $u\neq v$, $w,z\in S$ ($w$ might be equal to $z$), but no shortest directed $(u,v)$-path (of length two, in case $\oxc = \opsc$) contains $w$; in this case we say that $S$ is a \emph{transitive $\oxc$-coconvex set}.
    \end{itemize}

    If $D'\subseteq D$ is a source, or a sink, or a transitive strong component of $D$, then we say that $D'$ is an \emph{extreme} strong component of $D$. Consequently, if $S\subseteq V(D)$ is such that $D[S]$ a source (resp. sink, transitive) strong component, then $S$ is a source (resp. sink, transitive) $\oxc$-coconvex set of $D$.

    Thus, a direct consequence of Proposition~\ref{prop:coconvex_sets} is:
    \begin{cor}
    \label{cor:extremecomponentsarecoconvex}
            If $D$ is a oriented graph and  $D'\subseteq D$ is a source or a sink strong component of $D$, then any interval $\xc$-set and any hull $\xc$-set contains at least one vertex of $D'$, for $\xc \in \{\ogc,\opc, \opsc\}$.
        If $S$ is a transitive strong component of an oriented graph $D$, then any interval $\xc$-set and any hull $\xc$-set contains at least one vertex of $S$, for $\xc\in\{\ogc,\opsc\}$.
    \end{cor}

     
        It is well-known (see~\cite{BG2008}) that a digraph $D$ such that $V(D)\geq 2$ is strongly connected if, and only if, $D$ admits an ear (a.k.a. handle) decomposition which is a sequence $(P_0,P_1,\ldots, P_k)$, for some $k\geq 0$, such that: 
\begin{itemize}
    \item $P_0$ is a directed cycle in $D$;
    \item for every $i\in\{1,\ldots, k\}$, $P_i$ is a directed path of length at least one or a directed cycle in $D$, called an \emph{ear/handle}, such that if $P_i = (u,\ldots,v)$ is a directed path, then $V(P_i)\cap (V(P_0)\cup \ldots \cup V(P_{i-1})) = \{u,v\}$, otherwise $|V(P_i)\cap (V(P_0)\cup \ldots \cup V(P_{i-1}))| = 1$;
    \item The family $\{A(P_0),\ldots, A(P_k)\}$ is a partition of $A(D)$.
\end{itemize}
An ear $P_i$ containing a single arc is called \emph{trivial}, regardless if it is a loop or a path of length one, while other ears are \emph{non-trivial}.

To prove our first result, we need a stronger version of an ear decomposition based on shortest ears. The proof of the next lemma can be found in~\cite{BG2008}, Theorem~5.3.2.

\begin{lem}
\label{lem:short_ear_decomposition}
A directed graph $D$ such that $V(D)\geq 2$ is strongly connected if, and only if, it has an ear decomposition $(P_0,P_1,\dots,P_k)$ such that:
\begin{enumerate}
\item $P_0$ is a shortest directed cycle;
\item there exists $j\in \{0,\dots,k\}$ such that $P_1,\dots,P_j$ are non-trivial ears, while $P_{j+1},\dots,P_k$ are trivial;
\item if $P=(u,w,\dots,v)$ is non-trivial ear, then $P_{wv}$ is $(w,v)$-shortest path in $D$;
\end{enumerate}
where $P_{wv}$ denotes the subpath of $P$ from $w$ to $v$.
\end{lem}

    \section{(In)equalities for the oriented hull number}
    \label{sec:geodetic}

        In the sequel, we first present a general bound for strongly oriented graphs. Then, we determine exact values for hull number of a tournament in the geodetic and $\opsc$ convexities. In particular, these results imply  polynomial-time algorithms to compute $\ohng(T)$ and $\ohnps(T)$ when $T$ is a tournament. Afterwards, we focus on the case of $\opc$-convexity, which was already studied in the literature in the case of tournaments and extend some results to oriented graphs whose underlying graph is split or cobipartite. A \emph{tournament} is a oriented graph obtained from the orientation of each edge of a complete graph. Remember that we say that a graph $G=(V,E)$ is \emph{split} if we can partition $V=C\cup I$, where $C$ is a complete graph and $I$ is an independent set and a \emph{cobipartite graph} is a graph which is the complement of a bipartite graph. 


    \subsection{Geodetic and distance-two convexities}
    \label{subsec:strongly}

     \begin{prop}
         If $D$ is a strongly oriented graph, then $\ohng(D)\leq m(D)-n(D)+2$.
     \end{prop}
    \begin{proof}

Let $(P_0,\dots, P_k)$ be an ear decomposition as in Lemma \ref{lem:short_ear_decomposition}. Let us build a geodetic hull set $S$ of $D$ as follows: add to $S$ two vertices of the shortest cycle $P_0$ and, for each non-trivial ear $P_p$ = $(u_p, w_p,\dots, v_p)$, add $w_p$ to $S$, for each
$p\in\{1,\dots, j\}$. It is well-known that if  $(P_0,\dots, P_k)$ is an ear decomposition of $D$,
then $k = m(D)-n(D)$ as we have exactly one extra arc (compared to the number
of vertices) per ear $P_1,\dots,P_k$~\cite{BG2008}. Thus, it follows that $|S|\leq m(D) - n(D) + 2$,
as we pick exactly two vertices in the initial cycle, plus one vertex per remaining
non-trivial ear.

Let us claim that $S$ is indeed a geodetic hull set of $D$. Since $P_0$ is a shortest
cycle, $V(P_0)\subseteq\oifg(S)$ as we added two vertices say $x$ and $y$ of $P_0$ to $S$ and the
$(x, y)$-path and the $(y, x)$-path that form $P_0$ must be shortest directed paths.
Let $V_p = V(P_0) \cup \dots\cup V(P_p)$, for every $p \in \{0, \ldots , k\}$. Now, by induction
in $p \in \{1, \ldots , j\}$, note that $V_p \subseteq
\oifg^{p+1}
(S)$ because, by hypothesis, we know
that $V_{p-1} \subseteq\oifg^p(S)$, and we have that $P_p = (u_p, w_p,\dots, v_p)$ where $w_p \in S$, $v_p \in\oifg^p(S)$ and the $(w_p, v_p)$-subpath of $P_p$ is a shortest directed path in $D$, by Lemma~\ref{lem:short_ear_decomposition}. Consequently, $V(D) = V_0 \cup\dots\cup V_k =\ohullg(S)$ and the result follows.
\end{proof}
        
        
        

\begin{prop}
    \label{prop:coconvex_geo_paths}
        Let $D$ be an oriented graph and $P = (u,w_1,\ldots,w_q,v)$ be a directed not-shortest $(u,v)$-path in $D$, for some $q\geq 1$, such that $d_D^-(w_i)=d_D^+(w_i)=1$, for every $i\in\{1,\ldots,q\}$. Then, any minimum hull $\ogc$-set of $D$ contains at least one
        internal vertex of $P$.
    \end{prop}
    \begin{proof}
        First note that, since $P$ is not a shortest $(u,v)$-path and since $d_D^-(w_i)=d_D^+(w_i)=1$, for every $i\in\{1,\ldots,q\}$, we have that $W=\{w_1,\ldots, w_q\}$ is $\ogc$-coconvex as no directed shortest $(x,y)$-path contains a vertex of $W$, for any $x,y\notin W$. Thus, by Proposition~\ref{prop:coconvex_sets}, any hull $\ogc$-set of $D$ contains at least one vertex of $W$.
    \end{proof}

    \begin{cor}
        There is a strongly oriented graph $D$ such that $\ohng(D) = m(D)-n(D)$.
    \end{cor}
    \begin{proof}
        Let $D$ be the oriented graph obtained from an oriented triangle $(u,v,w)$ by the addition of $k\geq 2$ directed paths $P_i=(u,z_1^i,z_2^i,v)$, for each $i\in\{1,\ldots, k\}$.
        
        By Proposition~\ref{prop:coconvex_geo_paths}, note that at least one internal vertex of each $P_i$ must belong to any hull set of $D$, which is also sufficient to build a hull set of $D$.
    \end{proof}
        

    We now address to our results on tournaments. Let us first argue that, for a non-trivial strong tournament, two vertices always suffice to obtain a hull set.
    \begin{prop}
        \label{prop:strong-tourn-upbound}
        If $T$ is a strong non-trivial tournament, then: $$\ohnps(T) = \ohng(T) = 2.$$
    \end{prop}
    \begin{proof}
        First, suppose $T$ is a strong tournament. Since $T$ is strongly connected and non-trivial, then $T$ has at least three vertices and any hull $\opsc$-set (resp. $\ogc$-set) must contain at least two vertices.
        
        On the other hand, note that any hull $\opsc$-set is a hull $\ogc$-set, so it suffices to prove that $\ohnps(T) \leq 2$.

        Observe that any non-trivial strong tournament must contain directed cycles on three vertices.
        Let $u,v\in V(T)$ be such that $u,v$ lie in a directed $C_3$ and such that $|\ohullps(\{u,v\})|$ is maximized among all pairs of such distinct vertices $u,v\in V(T)$. Let $H = \ohullps(\{u,v\})$. We want to prove that $H = V(G)$. By contradiction, suppose that $\overline{H}= V(T)\setminus H \neq \emptyset$ and assume w.l.o.g. that $(u,v)\in A(T)$. 
        
        Since $u,v$ lie in a directed $C_3$, note that $H$ induces a strong sub-tournament of $T$. Indeed there exists a (shortest) directed $(v,u)$-path $P$ of length two in $T$ and any other vertex $w$ that is added to $H$ can be seen as the addition of an ear to the previous ear decomposition that we had before the inclusion of $w$, which started with the directed $C_3$ that contains $u$ and $v$.
        
        Let us prove that there is no vertex $w\in \overline{H}$ such that $N_T^-(w)\cap H\neq \emptyset$ and $N_T^+(w)\cap H\neq\emptyset$. Indeed, since $T$ is a tournament, any vertex of $H$ is either the head or the tail of an arc with $w$. Consequently, $\{N_T^-(w)\cap H\neq \emptyset, N_T^+(w)\cap H\neq\emptyset\}$ would be a partition of $H$ into two non-empty subsets and any edge linking such subsets must be oriented towards the endpoint in $N^+(w)\cap H\neq\emptyset\}$, as $H$ is $\opsc$-convex and $w\notin H$. Thus, this contradicts the fact that $H$ induces a strong sub-tournament of $T$.
        
        Consequently, $\overline{H}$ can be partitioned into two subsets $\overline{H}^+ = \{v\in V(T)\mid v\in \overline{H}\text { and } H\subseteq N_T^+(v)\}$ and $\overline{H}^- = \{v\in V(T)\mid v\in \overline{H}\text { and } H\subseteq N_T^-(v)\}$. Neither $\overline{H}^+$ nor $\overline{H}^-$ can be empty, as $T$ is strong. Note that the existence of an arc $(w,z)\in A(T)$ such that $w\in \overline{H}^-$ and $z\in \overline{H}^+$ would imply that $\ohullps(\{w,z\})\supseteq \{w,z\}\cup H$, contradicting the maximality of $H$, as $w,z$ would lie in a directed $C_3$ having a larger convex $\opsc$-hull. Thus, each edge having an endpoint in $\overline{H}^+$ and the other in $V(T)\setminus \overline{H}^+$ is oriented towards the one in $V(T)\setminus \overline{H}^+$, contradicting the hypothesis that $T$ is strong.
    \end{proof}

    Let us now argue that we just need to focus on strong tournaments, when computing not only the hull number, but also the interval number, of a tournament in the geodetic and $\opsc$ convexities.

    \begin{prop}
        \label{prop:hull_tournament_sum}
        If $T$ is a tournament, $D_1,\ldots, D_k$ are the strong components of $T$, for some positive integer $k$, and $\xc\in\{\ogc,\opsc\}$ then:
         $$\oinx(T) = \sum_{i=1}^k \oinx(D_i)\text{ and }\ohnx(T) = \sum_{i=1}^k \ohnx(D_i).$$
        In particular, if  there are $\ell \le k$ non-trivial strong components: $$\ohnps(T)=\ohng(T) = |Ext(T)|+2\ell.$$
    \end{prop}
    \begin{proof}Note that in a tournament $T$, all strong components are extreme, i.e. source, sink or transitive.
    By Corollary~\ref{cor:extremecomponentsarecoconvex}, we need to include at least one vertex of each of these components to any interval or hull set of $T$ in the $\ogc$ and $\opsc$ convexities. Futhermore, if $w$ lies in a shortest directed $(u,v)$-path with $u\neq w \neq v$ in $T$, then they all must belong to a same strong component of $T$. Thus, the result follows from Proposition~\ref{prop:strong-tourn-upbound} as two vertices always suffice for each non-trivial strong component.
    \end{proof}

    \subsection{Two-path convexity}
    \label{subsec:tournaments}

    The proof of Proposition \ref{prop:strong-tourn-upbound} can be used to show the same result for $\ohnp(T)$ of strong tournaments. 
    In~\cite{HW1996}, the authors study convex subsets of a given tournament $T$ in the $\opc$-convexity. As their focus is to obtain all convex subsets of a given tournament, they do a deep analysis on the structure of such convex subsets from which they deduce Proposition~\ref{prop:ohnp_torneio}. 
    
    \begin{prop}[\!\!\cite{HW1996}]
    \label{prop:ohnp_torneio}
        If $T$ is a non-trivial tournament, then $\ohnp(T)= 2$.
    \end{prop}
    
    
    
    Let us now observe that when computing $\ohnp(D)$ of a digraph $D$, we just need to pick at most two vertices of any subtournament $D'\subseteq D$. We actually prove something stronger:

\begin{prop}
\label{prop:hull_p3_subdigraph}
If $D'$ is a subdigraph of $D$ and $S$ is a minimum hull set of $D$ in the $\opc$-convexity, then $|S\cap V(D')|\leq \ohnp(D')$.
\end{prop}
\begin{proof}
The idea is that in the $\opc$-convexity having ``extra neighbors'' in $D$ does not change the fact that if $u,v$ belong to a set $S\subseteq V(D')$ and there is a directed $u,v$-path of length two passing through $w$ in $D'$, then $w\in \oifp(\{u,v\})$ not only in $D'$, but also in $D$. Thus, no \emph{minimum} hull set of $D$ may contain more than $\ohnp(D')$ vertices in $D'$.

Formally, by contradiction, let $D'$ be a subdigraph of $D$ and $S$ be a minimum hull set of $D$ in the $\opc$-convexity such that $|S\cap V(D')| > \ohnp(D')$.

Let $S'$ be a minimum hull set of $D'$ in the $\opc$-convexity, i.e. $|S'|= \ohnp(D')$. Define $S^* = (S\setminus(S\cap V(D'))\cup S'$.   Note that since $D'$ is a subdigraph of $D$, $S' \subseteq S^*$ and $S'$ is a hull set of $D'$ in the $\opc$-convexity, then $V(D')\subseteq \ohullp(S^*)$ and, as $S$ is a hull set of $D$, we deduce that, by the construction of $S^*$ is also a hull set of $D$ in the $\opc$-convexity. However, since $|S^*|<|S|$ as $|S\cap V(D')| > \ohnp(D') = |S'|$, we reach a contradiction to the minimality of $S$.
\end{proof}

    As a direct consequence of Propositions~\ref{prop:ohnp_torneio} and~\ref{prop:hull_p3_subdigraph}, we have the following.
    
    \begin{cor}
    \label{cor:ohnp_atmosttwo_subtournament}
    If $T$ is a tournament and $T$ is a subdigraph of an oriented graph $D$, then, for any minimum hull set $S$ of $D$ in the $\opc$-convexity, we have that $|S\cap V(T)|\leq 2$.
    \end{cor}
    
    We may derive some straightforward consequences of this fact. 
    
    \begin{cor}
        \label{cor:complement_k_partite}
        If $D$ is an oriented graph whose underlying graph is the complement of a $k$-partite graph, then $\ohnp(D)\leq 2k$. In particular, if $D$ is an oriented graph whose underlying graph is cobipartite, then $\ohnp(D)$ can be computed in polynomial time.
    \end{cor}
    
     
    
    \begin{cor}
    If $D$ is an oriented graph whose underlying graph is split, then $\ohnp(D)$ can be computed in polynomial time.
    \end{cor}
    \begin{proof}
        Let $\{S,K\}$ be a split partition of $V(G)$ such that $K$ is a maximal clique and $S$ is an independent set. Note that such partition can be obtained in linear time.
        
        Let $S_s$ and $S_t$ be the sets of sinks and sources of $G$ in $S$, respectively.
        Let us prove that if $X$ is a minimum hull set of $D$ in the $\opc$-convexity, then $X$ has exactly the vertices in $S_s\cup S_t$ plus at most two other vertices.
        
        If $|K| = 1$, then each vertex of $S$ is either a sink or a source and the result trivially follows. Suppose then that $|K|\geq 2$. By Proposition~\ref{prop:ohnp_torneio}, let $\{u,v\}$ be a minimum hull set of $D[K]$. By the proof of Proposition~\ref{prop:hull_p3_subdigraph}, note that $K\subseteq \ohullp(\{u,v\})$ in $D$. Moreover, each vertex that is neither a sink nor a source of $D$ in $S$ must have at least one in-neighbor and at least one out-neighbor in $K$. Consequently, each vertex that is neither a sink nor a source vertex also belongs to $\ohullp(\{u,v\})$. Thus $S_s\cup S_t\cup \{u,v\}$ is a hull set in the $\opc$-convexity of $D$.
        
        By Proposition~\ref{prop:coconvex_sets}, we know that $S_s\cup S_t$ must belong to any hull set of $D$ in the $\opc$-convexity. Thus, it remains to check whether at most two other vertices are needed, which can be done in polynomial time.
    \end{proof}

    \section{Hardness results}
    \label{sec:hardness}
    
    \subsection{A meta-theorem for computational complexity}
    \label{subset:meta-theorem}
    
    Let $D$ be an oriented graph. Let $B_D$ be the oriented graph such that $B_D$ has two vertices $v^i$ and $v^o$ for each vertex $v\in V(D)$ and $B_D$ has the following arcs: $(v^i,v^o)$, for each $v\in V(D)$; $(u^o,v^i)$ , for each $(u,v)\in A(D)$.

    \begin{prop}
    \label{prop:equivalentBipartiteInstances}
        For any connected oriented graph $D$ such that $n(D)\geq 2$, we have $\oing(D)=\oing(B_D)$ and $\ohng(D)=\ohng(B_D)$.
    \end{prop}
    \begin{proof}
        Let $P = (u,w_1,\ldots,w_q,v)$ be a directed shortest $(u,v)$-path in $D$, for some $q\geq 0$. Note that:
        \begin{itemize}
            \item $P^{ii} = (u^i,u^o,w_1^i,w_1^o, \ldots, w_q^i,w_q^o,v^i)$,
            \item $P^{io} = (u^i,u^o,w_1^i,w_1^o, \ldots, w_q^i,w_q^o,v^i,v^o)$,
            \item $P^{oi} = (u^o,w_1^i,w_1^o, \ldots, w_q^i,w_q^o,v^i)$ and
            \item $P^{oo} = (u^o,w_1^i,w_1^o, \ldots, w_q^i,w_q^o,v^i,v^o)$
        \end{itemize}
        are shortest directed paths in $D_B$ that we call \emph{corresponding paths} in $D_B$ to $P$. Note that any directed path in $D_B$ containing $w^i$ as internal vertex must also contain $w^o$ and vice-versa, for any $w\in V(D)$.
        
        Let us first prove that if $S_B$ is an interval (resp. hull) $\ogc$-set of $D_B$, then $S = \{v\in D\mid v^o\in S_B\text{ or }v^i\in S_B\}$ is an interval (resp. hull) $\ogc$-set of $D$. Let $w\in V(D)$ such that $w\notin S$. It implies that $w^o,w^i\notin S_B$. Thus, there exists $u^x,v^y\in S_B$ (resp. $u^x,v^y\in \oifg^k(S_B)$, for some $k\geq 0$) for some $x,y\in\{i,o\}$, and a shortest $(u^x,v^y)$-path $P$ such that $w^o,w^i\in V(P)$. Consequently, $u,v\in S$ (resp. $u,v\in \oifg^k(S)$, for some $k\geq 0$) and $w$ belongs to a shortest $(u,v)$-path in $D$. Thus, $S$ is an interval (resp. hull) $\ogc$-set of $D$.
        
        Conversely, suppose now that $S$ is an interval (resp. hull) $\ogc$-set of $D$. Let $S_B = \{f(v)\mid v\in S\}$ where $f:V(S)\to V(D_B)$ is define as follows:
        
        $$f(v)=\begin{cases}
        v^i &\text{, if } v\text{ does not belong to a sink strong component of }D;\\
        v^o &\text{, otherwise. }
        \end{cases}$$
        
        Let $w^x\in V(D_B)$ such that $w^x\notin S_B$, for some $x\in\{i,o\}$. Suppose first that $w\notin S$. It implies that $w^o,w^i\notin S_B$.  Since $S$ is an interval (resp. hull) $\ogc$-set of $D$, there exists $u,v\in S$ (resp. $u,v\in \oifg^k(S)$, for some $k\geq 0$), and a shortest $(u,v)$-path $P$ such that $w\in V(P)$. Thanks to the corresponding paths in $B_D$, we deduce that $w\in \oifg(S)$ (resp. $w\in \oifg^{k+1}(S)$).
        
        Suppose now that $w\in S$, $w^x\in S$, for some $x\in\{i,o\}$, but $w^y\notin S$, where $y = \{i,o\}\setminus \{x\}$. We need to argue that $w^y\in \oifg(S)$ (resp. $\ohullg(S)$). If $D$ is strongly connected, then $n(D)\geq 3$ and thus $|S|\geq 2$. Let $u^z\in S_B\setminus \{w^x\}$. Since $D$ is strongly connected, there is a (shortest) directed $(w^i, u^z)$-path and a (shortest) directed $(u^z,w^0)$-path. This implies that $w^y\in \oifg(S)$ (resp. $\ohullg(S)$), regardless $y=i$ or $y=o$.
        
        Suppose then that $D$ is not strongly connected. Suppose first that $w$ does not belong to a sink strong component of $D$. It means that $w^i\in S$. By Proposition~\ref{prop:inoutsections}, there is a sink strong component $C'$ of $D$ such that $w_i\notin C'$ and $C'\subseteq S_D^+(w^i)$. By Corollary~\ref{cor:extremecomponentsarecoconvex}, there is a vertex $u^z\in C'\cap S_B$, for some $z\in\{i,o\}$. Consequently, $w^o$ lies in any (shortest) directed $(w^i,u^z)$-path. Analogously, suppose that $w$ belongs to a sink strong component $C$. Thus, $w^o\in S_B$ and, by Proposition~\ref{prop:inoutsections}, there is a source strong component $C'$ of $D$ such that $w_o\notin C'$ and $C'\subseteq S_D^-(w^o)$.  By Corollary~\ref{cor:extremecomponentsarecoconvex}, there is a vertex $u^z\in C'\cap S_B$, for some $z\in\{i,o\}$. Consequently, $w^i$ lies in any (shortest) directed $(u^z,w^o)$-path.
    \end{proof}
    \begin{cor}
        \label{cor:hardnessBipartite}
        Given an oriented graph $D$, if deciding whether $\oing(D)\leq k$ or $\ohng(D)\leq k$ is $\NP$-hard or $\W[i]$-hard parameterized by $k$, for some $i\in\mathbb{Z_+^*}$, then the same holds even if the underlying graph of $D$ is bipartite.
    \end{cor}
    \begin{proof}
    Note that $D_B$ is an oriented bipartite graph, for any oriented graph $D$, as each cycle doubles its length. Thus, Proposition~\ref{prop:equivalentBipartiteInstances} provides the required reduction.
    \end{proof}
    
    By~\cite{AA2021}, we already know that deciding whether $\oing(D)\leq k$ is $\W[2]$-hard parameterized by $k$, and that deciding whether $\ohng(D)\leq k$ is $\NP$-hard, even if the underlying graph of $D$ is bipartite. Thus, we believe that Corollary~\ref{cor:hardnessBipartite} might be useful to prove $\W$-hardness for the oriented hull number, or even to prove hardness results for subclasses of bipartite graphs.
    
    \subsection{Hardness of hull number in  two-path and distance-two convexities}
    \label{subsec:hardness-hull}
    
    \newcommand{\probohnp}{\pname{Parameterized oriented $\opc$-hull}\xspace}

        \paraProblemDef{\probohnp}{Oriented graph $D$ and positive integer $k$.}{$k$}{Is $\ohn_{P_3}(D)\leq k$?}{pro:ohnp}

         \begin{thm}
         \label{thm:npohnpbipartite}
         \nameref{pro:ohnp} is \W[2]-hard even if $D$ is acyclic and its underlying graph is bipartite.
        \end{thm}
     \begin{proof}
         We modify a reduction done in~\cite{Nichterlein2013} from the \textsc{Parameterized hitting set} problem to show that \nameref{pro:ohnp} is is $W[2]$-hard in undirected bipartite graphs.   
        
        Given a set $U=\{u_1, \dots, u_n\}$, a family of subsets $\mathcal{W} = \{W_1, \dots, W_m\}$ over the elements of $U$ and an integer $k' \geq 0$, the \textsc{Parameterized hitting set} problem consists in deciding whether there is $U' \subseteq U$ such that $|U'| \leq k'$ and $U' \cap W_j \neq \emptyset$, for every $j \in [m]$. In the positive cases, we say that $U'$ is a \emph{hitting set} of $\mathcal{W}$. This problem is $W[2]$-hard with respect to the parameter $k'$ \cite{DF12}.
        
         We are going to construct an acyclic oriented graph $D_I = (V, A)$ from an instance $I = (\mathcal{W}, U, k')$ of \textsc{Parameterized hitting set} so that $\ohnp(D_i)\leq k=k'+2$ if and only if $I$ is a positive instance, that is, if there is a hitting set $U'\subseteq U$ of $\mathcal{W}$ with at most $k'$ elements. 
         
         Suppose that initially $V(D_I) = \emptyset$ and $A(D_I) = \emptyset$. 
         For each $W_j \in \mathcal{W}$, we add a vertex $w_j$ to $V(D_I)$, for $1 \leq j \leq m$. 
        
         For each $u_i \in U$, $1 \leq i \leq n$, let $s(i)$ be the number of sets of $\mathcal{W}$ which $u_i$ belongs, and let us call $w_i^1, \dots, w_i^{s(i)}$ the corresponding vertices of such sets (remark that we do not create new vertices to represent the sets of $\mathcal{W}$). We construct two gadgets $G_i^1$ and $G_i^2$ to represent $u_i$ and its relation with the sets of $\mathcal{W}$ as follows (see Figure~\ref{envoltoria_w[2]_dificil_DAG}). 

         \begin{figure}[h!]
            \centering
            \includegraphics[scale=0.4]{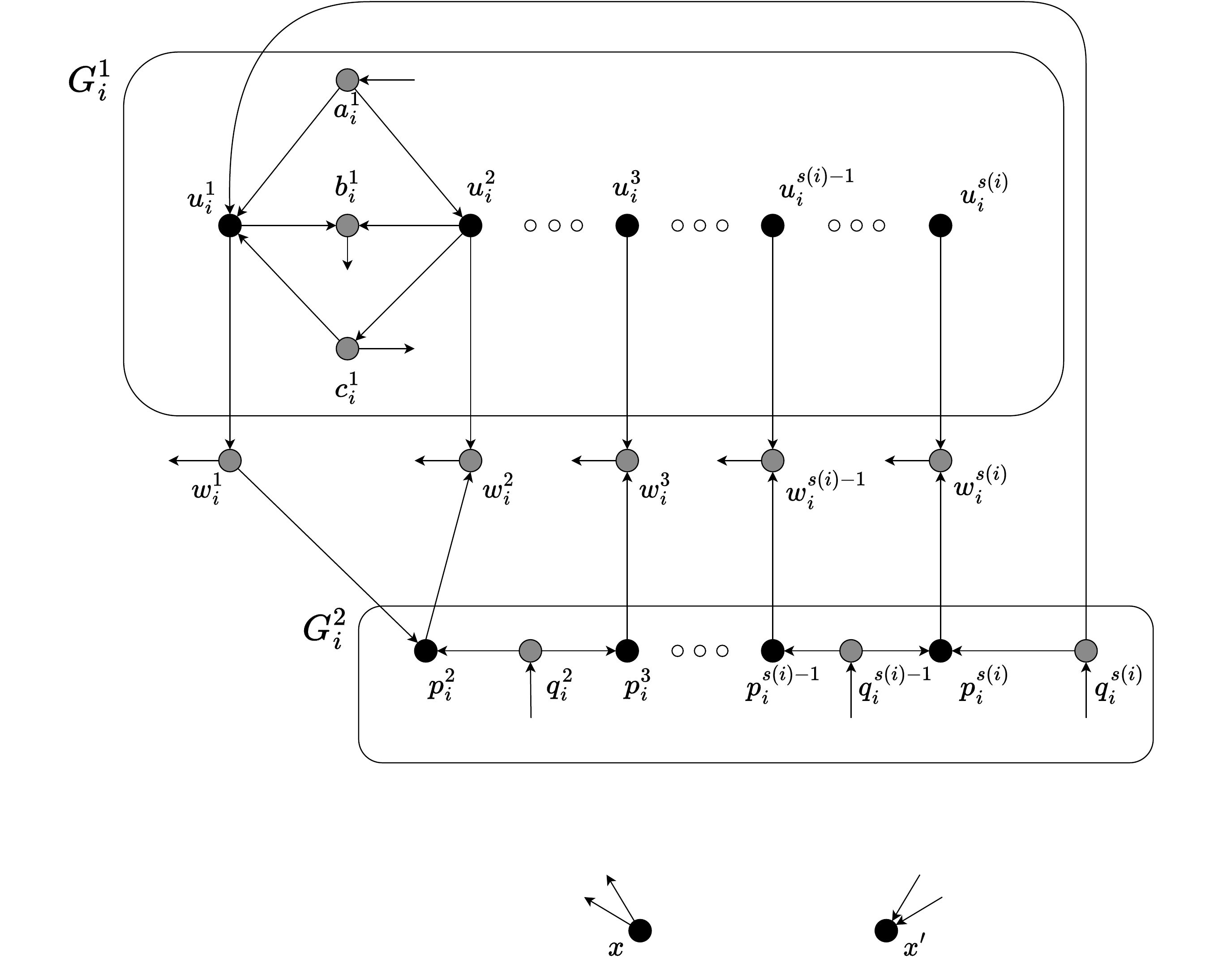}
            \caption{Gadget for $u_i \in U$.}
            \label{envoltoria_w[2]_dificil_DAG}
        \end{figure}
         
         For $G_i^1$, we add the vertices $u_i^1, \dots, u_i^{s(i)}$ to $V(D_I)$ and vertices $a_i^j, b_i^j, c_i^j$, $1 \leq j \leq s(i)-1$.  To conclude the construction of $G_i^1$, we add the arcs $(a_i^j,u_i^j)$,  $(a_i^j,u_i^{j+1})$, $(u_i^j,b_i^j)$,  $(u_i^{j+1},b_i^j)$, $(c_i^j,u_i^j)$, $(u_i^{j+1},c_i^j)$, for each $1 \leq j \leq s(i)-1$. 
        
         For $G_i^2$, we add the vertices $p_i^2, \dots, p_i^{s(i)}$ and $q_i^2, \dots, q_i^{s(i)}$ to $V(D_I)$, with the arcs $(q^j_i,p_i^j)$ for $2\leq j\leq s(i)$ and the arcs $(q^j_i, p^{j+1}_i)$ for $2\leq j\leq s(i)-1$.   
        We then connect $G_i^1$, $G_i^2$ and $w_i^1, \dots , w_i^{s(i)}$ by adding the arcs $(u_i^j,w_i^j)$, for $1 \leq j \leq s(i)$, $(p_i^j,w_i^j)$ for $2 \leq j \leq s(i)$ and the arcs $(w_i^1,p_i^2)$ and $(q_i^{s(i)},u_i^1)$ to $A(D_I)$. 
        
        Finally, we add the vertices $x,x'$ to $V(D_I)$, and arcs from:
    \begin{itemize}
        \item $b^j_i,c^j_i$ to $x'$, for each $1\leq i\leq n$ and $1\leq j\leq s(i)$;
        \item $w_j$ to $x'$, for each $1\leq j \leq m$;
        \item $x$ to $a_i^j$, for each $1\leq i\leq n$ and $1\leq j\leq s(i)-1$;
        \item $x$ to $q_i^j$, for each $1\leq i\leq n$ and $2\leq j\leq s(i)$. 
    \end{itemize}


        Note that, by construction, $D_I$ has no directed cycles and its underlying graph is bipartite (vertices $x$, $x'$, $u_i^j$ and $p_i^j$ belonging to one part and $a_i^j$, $b_i^j$, $b_i^j$, $w_i^j$ to the other. See Figure~\ref{envoltoria_w[2]_dificil_DAG}).         
        We now show that, given an instance $I$ of \textsc{Parameterized hitting set}, there is a hitting set $U'$ of $\mathcal{W}$ with size at most $k'$ if and only if $D_I$ has a hull set $S$ of size $k = k' + 2$. 

\begin{myclaim}\label{afirmacao_red1}
    Let $S\subseteq V(D_I)$ be such that $x,x'\in S$. If $u_i^1\in \ohullp(S)$, then $V(G_i^1)\subseteq \ohullp(S)$. 
\end{myclaim}
\begin{claimproof}
Let $l\geq 1$ be the smallest integer $l$ such that $u^1_i\in I^l(S)$. Since $x,x'\in S$, considering the adjacency of $a^1_i$ and $b^1_i$, such vertices are appear as internal vertices of paths of length two between $(u^1_i,x')$ or $(x,u^1_i)$, implying that $a^1_i,b^1_i\in I^{l+1}(S)$. Since $u^2_i$ is also an internal vertex of a $(a^1_i,b^1_i)$-path of length two, $u^2_i\in I^{l+2}(S)$. Finally, $c^1_i$ is an internal vertex of a path of length two between $u^1_i$ and $u^2_i$, so $c^1_i\in I^{l+2}(S)$. Repeating this argument for $u^j_i$, $2 \leq j \leq s(i)$, we conclude that every vertex of $G_i^1$ belongs to $\ohullp(S)$. 
\end{claimproof}

\begin{myclaim}\label{afirmacao_red2}
   Let $S\subseteq V(D_I)$ be such that $x'\in S$. If $V(G_i^1)\subseteq \ohullp(S)$, then $w_i^1,\dots,w_i^{s(i)}$ $\in \ohullp(S)$.
\end{myclaim}
\begin{claimproof}
    Observe that each $w^j_i$ is an internal vertex of a $(u^j_i,x')$-path of length two, for each $1\leq j\leq s(i)$.  So, $w_i^1,\dots,w_i^{s(i)}\in \ohullp(S)$.
\end{claimproof}

\begin{myclaim}\label{afirmacao_red3}
    Let $S\subseteq V(D_I)$ be such that $x\in S$. If $w_i^1,\dots,w_i^{s(i)}\in \ohullp(S)$, then $V(G_i^2)\subseteq \ohullp(S)$.
\end{myclaim}
\begin{claimproof}
    Since $w_i^1,w_i^2\in \ohullp(S)$, we have that $p_i^2\in \ohullp(S)$.
    Also, since $x,p_i^2\in \ohullp(S)$, we conclude that $q_i^2\in \ohullp(S)$. Similarly, as $q_i^2,w_i^3\in \ohullp(S)$, we have that $p_i^3\in \ohullp(S)$.
    Applying the above argument to $p_i^j,q_i^j\in \ohullp(S)$, for each $3\leq j\leq s(i)$ and, we conclude that $V(G_i^2)\subseteq \ohullp(S)$.
\end{claimproof}

$(\Rightarrow)$ Let $U'$ be a transversal of $\mathcal{W}$ with at most $k'$ elements. We are going to construct a hull set $S$ on the $\opc$-convexity  for $D_I$ with at most $k = k' + 2$ elements. We start by making $S=\{x,x'\}$, since $x$ and $x'$ are extremal vertices. For each $u_i\in U'$, we add the vertex $u_i^1$ to $S$. Then, $|S|\leq k'+2 = k$. We claim that $S$ is a hull set 
 for $D_I$. For each $u_i\in U'$, by Claim~\ref{afirmacao_red1}, we have that $V(G_i^1)\subseteq \ohullp(S)$. By Claim~\ref{afirmacao_red2}, we have that $w_i^j\in \ohullp(S)$, for each $u_i\in U'$ and for each $1\leq j\leq s(i)$. However, since $U'$ is a transversal of $\mathcal{W}$,  $w_j\in \ohullp(S)$, for each $1\leq j\leq m$. By Claim~\ref{afirmacao_red3}, we can verify that $V(G_i^2)\subseteq \ohullp(S)$, para todo $1\leq i\leq n$. In particular, $q_i^{s(i)} \in \ohullp(S)$, for every $1\leq i\leq n$.
As $q_i^{s(i)}, w_i^1\in \ohullp(S)$ even for $u_i$ such that $u_i\notin U'$, we conclude that $u_i^1$ corresponding to such elements are in $\ohullp(S)$, since $u^1_i$ is an internal vertex of a $(q_i^{s(i)},w_i^1)$-path of length two, for every $1\leq i\leq n$. Therefore $V(G^1_i)\subseteq \ohullp(S)$, for every $1\leq i\leq n$. Consequently, $S$ if indeed a hull set of $D_I$.

$(\Leftarrow)$ Now suppose there is a hull set $S$ on $\opc$ convexity for $D_I$ of size $|S|\leq k = k'+2$. Observe that $x',x$ are, respectively, a source and a sink in $D_I$. So, such vertices must be in $S$.

First, if $v\in V(G_i^1)\setminus\{u_i^1\}$ belongs to $S$, then $S' = (S\setminus\{v\})\cup\{u_i^1\}$ is also a hull set for $D_I$ by Claim~\ref{afirmacao_red1}. Therefore, we can suppose without loss of generality that $ u_i^1 \in S\cap V(G_i^1)$, for every $1\leq i\leq n$.

If $w_j$ belongs to $S$, for some $1\leq j\leq m$, then $S' = (S\setminus\{w_j\})\cup\{u_i^1\}$, for some $i\in\{1,\ldots, n\}$ such that $u_i\in W_j$, is also a hull be of $D_I$, by Claim~\ref{afirmacao_red2}. So, $S\cap \{w_1,\ldots, w_m\} = \emptyset$.

Finally, if $v\in V(G_i^2)$ belongs to $S$, for some $1\leq j\leq m$, note that $S' = (S\setminus\{v\})\cup\{u_i^1\}$ is also a hull set of $D_I$ by Claims~\ref{afirmacao_red1}, \ref{afirmacao_red2} and \ref{afirmacao_red3}. Therefore, $S\cap V(G_i^2) = \emptyset$, for every $1\leq i\leq n$. So, we have that $S'\setminus\{x,x'\}\subseteq \{u_i^1\mid 1\leq i\leq n\}$, without loss of generality.

So, if $S$ is a hull set of $D_I$, there is another hull set $S'$ of same size such that $S'\setminus\{x,x'\}\subseteq \{u_i^1\mid 1\leq i\leq n\}$.


We argue that the set $U' = \{u_i\in U\mid u_i^1\in S'\}$ is a \emph{transversal} of $\mathcal{W}$. Since $x'$ and $x$ belongs to $S$, note that $|U'|\leq k-2=k'$.


Suppose, by contradiction, that there is $W_l\in \mathcal{W}$ such that $U'\cap W_l=\emptyset$. This means that none of $u^j_i$, such that $w^j_i\in N^+(u^j_i)$, belongs to $S'$. Recall that there are no vertices $p^j_i$, $q^j_i$, $a^j_i$, $b^j_i$ or $c^j_i$ in $S'$. So, $u^j_i$ can only be generated from a path of length two containing $w^j_i$ as an end-vertex. 

Let $u^j_i$ be a vertex added to $\ohullp(S')$ and $w^j_i\in N^+(u^j_i)$, for all $1\leq j\leq s(i)$. Since $w^j_i$ was added to $\ohullp(S')$ before $u^j_i$ in $\oifp^p(S')$, for some $p\geq 1$. However, $w^j_i$ was included in $\ohullp(S')$ as a result of $p^j_i$ and $x'$, say in $\oifp^q(S')$. If $p^j_i$ was added to $\oifp^{q-1}(S')$ as a consequence of $w^j_i$ and $q^{j-1}_i$. Contradiction to the construction of $D_I$.
   \end{proof}
    
     Since bipartite graphs have no cycle on three vertices, the same holds for the $\opsc$ convexity.
    \newcommand{\probohnps}{\pname{Parameterized oriented $\opsc$-hull}\xspace}

        \paraProblemDef{\probohnps}{Oriented graph $D$ and positive integer $k$.}{$k$}{Is $\ohnps(D)\leq k$?}{pro:ohnps}

        \begin{cor}
        \label{cor:npohnpsbipartite}
        \nameref{pro:ohnps} is \W[2]-hard even if $D$ is acyclic and its underlying graph is bipartite.
    \end{cor}

\begin{figure}[h!]
    \centering
    \includegraphics[scale=0.4]{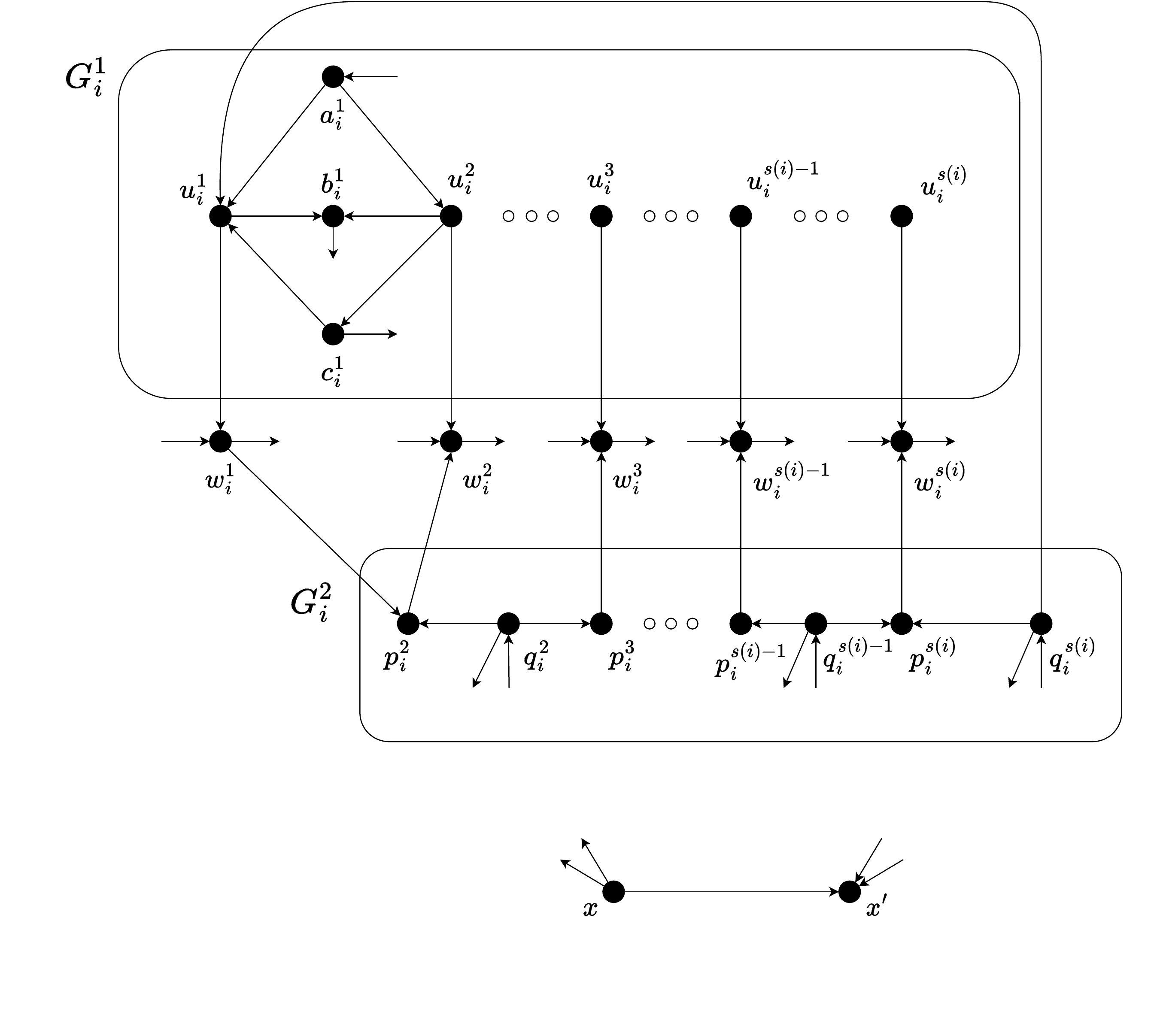}
      \caption{Gadget for $u_i\in U$.}
    \label{envoltoria_w[2]_dificil_g_DAG}
\end{figure}

Making a small change to the construction shown in Figure \ref{envoltoria_w[2]_dificil_g_DAG}, we were able to show that the problem of given an integer $k$ and an oriented acyclic graph, to determine if $\ohng\leq k$ is $W[2]$-hard.

\paraProblemDef{PARAMETERIZED ORIENTED $\ogc$-HULL}{Oriented graph $D$ and positive integer $k$.}{$k$}{Is $\ohng(D)\leq k$?}{pro:ohng}

\begin{cor}\label{corohngdag}
     \nameref{pro:ohng} is $W[2]$-hard even if $D$ is acyclic. 
\end{cor}
\begin{proof}
    The proof is analogous to the one presented in the Theorem~\ref{thm:npohnpbipartite}, thus we omit the details. See Figure~\ref{envoltoria_w[2]_dificil_g_DAG} to observe the minor changes in the construction: an arc from $x$ to $x'$, arcs from $q_i^j$ to $x'$, for every $j\in\{2,\dots, s(i)\}$ and from $x$ to $w_i^j$, for every $j\in \{1,\dots, s(i)\}$.
\end{proof}

    \subsection{Hardness of interval number in two-path and distance-two convexities}
    \label{subsec:hardness-interval}

    A graph $G$ is \emph{chordal bipartite} if $G$ is bipartite and any chordless cycle of $G$ has length four. Note that, although the name is misleading, this class of graphs is not a subclass of chordal graphs.

    A \emph{dominating set} $S\subseteq V(G)$ in a graph $G$ satisfies that each $v\in V(G)\setminus S$ has a neighbor in $S$. The \emph{domination number} of $G$, denoted by $\gamma(G)$, is the cardinality of a minimum dominating set of $G$.

\problemDef{\pname{Dominating Set}}{Graph $G$ and positive integer $k$.}{Is $\gamma(G)\leq k$?}{pro:DS}

We know that the \nameref{pro:DS} problem is NP-complete for chordal bipartite graphs~\cite{MB87}. We reduce such problem to \nameref{pro:oinp}.

\problemDef{\pname{Oriented $\opc$-interval}}{Oriented graph $D$ and positive integer $k$.}{Is $\oinp(D)\leq k$?}{pro:oinp}

\begin{thm}
\label{prop:oinp}
 \nameref{pro:oinp} is \NP-complete, even if the input oriented graph $D$ is a DAG and its underlying graph is chordal bipartite.
\end{thm}

\begin{proof}
The proof is similar to~\cite{CDS09}. It is known that the \nameref{pro:DS} problem is \NP-complete even if the input graph is bipartite chordal~\cite{MB87}.

Note that \nameref{pro:oinp} problem belongs to \NP~because given as certificate a subset of vertices $S$ of an oriented chordal bipartite graph $D$, it is possible to determine in linear time if it has the appropriate size and if every vertex $v \in V(G) \setminus S$ there are arcs $(v_1,v)$ and $(v,v_2)$ for some $v_1,v_2\in S$.

Let $I=(G,k)$ be an arbitrary instance of \nameref{pro:DS} such that $G=(A\cup B,E(G))$ is chordal bipartite with $n$ vertices.

Let us build an instance $(D, k+n)$ of \nameref{pro:oinp} as follows. First add to $D$ all vertices and edges of $G$ and orient all edges from $A$ to $B$. Now, for each $v \in V(G)$ we create in $D$ an auxiliary vertex $z_v$ and add the edge $vz_v$. If $v\in A$ orient the arc $(z_v,v)$ and if $ v\in B$ then orient $(v,z_v)$. Let $Z = \{z_v\mid v\in V(D)\}$.

Note that the underlying graph of $D$ is bipartite chordal graph as the addition of vertices of degree one does not create cycles. Moreover, observe that $D$ has no directed cycles and thus it is a DAG.
Let us prove that $\gamma(G)\leq k$ if, and only if, $\oinp(D)\leq k+n$.

Suppose that $S$ is a dominating set of $G$ with at most $k$ vertices. We claim that $U=S\cup Z$ is an interval $\opc$-set of $D$. Since $|Z| = n$ and $|S|\leq k$, we then have $\oinp(D)\leq k+n$. Let $v\in V(D)$ such that $v\notin U$. Since $Z\subseteq U$, note that $v\in A\cup B$. Since $S$ is a dominating set, there exists $x\in S$ such that $vx\in E(G)$. If $v \in A$, then $(z_v,v),(v,x)\in A(D)$ and $z_v,x\in U$. Thus $v\in \oifp(U)$. Similarly, if $v \in B$, there are arcs $(x,v),(v,z_v)\in A(D)$ such that $x,z_v\in U$ implying that $v\in \oifp(U)$.

Reciprocally, let $U$ be a interval $\opc$-set with cardinality at most $n+k$. By Proposition~\ref{prop:coconvex_sets}, note that $Z\subseteq U$. Consequently, $S = U\setminus Z \subseteq A\cup B$ has at most $k$ vertices. Since $U$ is an interval $\opc$-set, note that every vertex $v\in A\setminus U$ must have an out-neighbor $x\in B\cap U$. Similarly, every vertex $v\in B\setminus U$ must have an in-neighbor $x\in A\cap U$. Consequently, $S$ is a dominating set of $G$.
\end{proof}

\problemDef{\pname{Oriented $\opsc$-interval}}{Oriented graph $D$ and positive integer $k$.}{Is $\oinps(D)\leq k$?}{pro:oinps}

\begin{cor}
\label{cor:oinps_chordal_bip_np}
 \nameref{pro:oinps} is \NP-complete, even if the input oriented graph $D$ is a DAG and its underlying graph is chordal bipartite.
\end{cor}
    
In~\cite{AA2021}, the authors prove that deciding whether $\oing(D)\leq k$ is \NP-hard even for split graphs. Computing $\inp(G)$ is $\NP$-hard for a chordal $G$~\cite{CDS09}.

\paraProblemDef{\pname{Parameterized Set Cover}}{Set $U, \mathcal{F}\subseteq\mathcal{P}(U)$ and $k\in\mathbb{Z}_+^*.$}{$k$}{Does there exist $\mathcal{F'}\subseteq\mathcal{F}$ s.t. $\bigcup \mathcal{F'}=U$ and $|\mathcal{F'}|\leq  k$?}{pro:SetCover}

\newcommand{\proboinpparam}{\pname{Parameterized oriented $\opc$-interval}\xspace}

\paraProblemDef{\proboinpparam}{Oriented graph $D$ and positive integer $k$.}{$k$}{Is $\oinp(D)\leq k$?}{pro:oinpparam}

\begin{prop}
\label{prop:oinps_w-hard_split}
 \nameref{pro:oinpparam} is \W[2]-hard, even if $D$ has an underlying split graph.
\end{prop}
\begin{proof}
It is well known that \nameref{pro:SetCover} is $\W[2]$-hard~\cite{Cyganetal2015}.
We reduce the \nameref{pro:SetCover} to  \nameref{pro:oinpparam}. Let~$\mathcal{I}=(U, \mathcal{F}, k)$ be an input  to \nameref{pro:SetCover} such that $U=\{1,2,\dots,n\}$ and $\mathcal{F}=\{F_1,F_2,\dots,F_m\}$. 

First note that each element in $U$ must belong to at least one subset $F_i\in \mathcal{F}$ as otherwise the instance is trivially a NO-instance. Thus, we assume $\cup_{F\in \mathcal{F}}F=U$.

Consequently, if $|U|\leq k$, then the instance is trivially a YES-instance as we may arbitrarily choose exactly one element of $\mathcal{F}$ for each element in $U$. Similarly, if $|U|=k+1$, one can easily answer in linear time whether $(U,\mathcal{F},k)$ is a YES-instance to \nameref{pro:SetCover} by just checking whether there is an element in $\mathcal{F}$ containing at least two elements in $U$. Thus, also assume that $|U|\geq k+2$.

We will construct an oriented graph $D$ such that $(\mathcal{U},\mathcal{F},k)$ is a YES-instance if and only if $\oinp(D)\leq k+2$. 

In $D$ we have as vertex set $$V(D) = \{u_j\mid j\in \{1,\ldots, n\}\}\cup \{f_i\mid i\in \{1,\ldots, m\}\}\cup \{v,w\}.$$
Let us now define the arc set of $D$. First we add the arcs $(f_i,u_j)$ if $j\in F_i$, for every $j\in\{1,\ldots, n\}$ and $i\in\{1,\ldots, m\}$. Then, we add the arcs $(f_i,w)$ and $(v,f_i)$, for every $i\in \{1,\dots,m\}$. We construct a transitive tournament induced by $\{f_i\mid i\in\{1,\ldots,m\}\}$ by adding the arcs $(f_i,f_{i'})$ whenever $1\leq i<i'\leq m$. Finally, we include the arcs $(u_j,v)$ for every $j\in \{1,\dots, n\}$.

Note that by construction $C=\{v\}\cup \{f_i\mid i\in \{1,\ldots,m\}\}$ is a clique and $S=\{w\}\cup \{u_j\mid j\in\{1,\ldots,n\}\}$ is a stable set, therefore the underlying graph of $D$ is split. 

Let us now prove that $\mathcal{I}$ is an YES-instance if and only if $\oinp(D)\leq k+2$.

Suppose first that $\mathcal{I}$ is an YES-instance and let $\mathcal{F'}\subseteq \mathcal{F}$ be such that $\mathcal{F'}=\{F_i\mid i\in I\}$, $\bigcup \mathcal{F'} = U$ and $|I|\leq k$. Let $S\subseteq V(D)$ be such that $S = \{v,w\}\cup \{f_i\mid i\in I\}$. We claim that $S$ is an interval set for $D$ in the $\opc$-convexity. Indeed, note that all vertices in $\{f_i\mid i\in\{1,\ldots,m\}\}$ belong to a directed $(v,w)$-path on three vertices in $D$. Moreover, since $\mathcal{F'}=\{F_i\mid i\in I\}$ covers all elements in $U$, note that, for every $u\in\{u_j\mid \{1,\ldots,n\}\}$, $u$ belongs to a directed $(f_i,v)$-path for some $f_i\in S$ as there is $F_i\in \mathcal{F}$ such that $u\in F_i$.

Suppose now that $S$ is an interval set for $D$ in the $\opc$-convexity such that $|S|\leq k+2$.
Since $w$ is a sink, it belongs to every interval set in $\opc$-convexity by Proposition~\ref{prop:coconvex_sets}. Thus, $w\in S$ and consequently, at most $k+1$ vertices of $S$ belong to $U$. By hypothesis, we have that $|U|\geq k+2$. Consequently, there is at least one vertex of $U$ that does not belong to $S$. Since $N_D^+(u_j)=\{v\}$ and $S$ is an interval set of $D$ in the $\opc$-convexity, we deduce that $v\in S$. Since $v,w\in S$, note that any vertex in $\{f_i\mid i\in\{1,\ldots,m\}\}$ belongs to a directed $(v,w)$-path. Consequently, we can build an interval set $S'$ from $S$ such that $S'\subseteq \{f_i\mid i\in\{1,\ldots, m\}\}\cup\{v,w\}$ by exchanging each vertex $u\in S\cap \{u_j\mid j\in\{1,\ldots, n\}\}$ by a vertex $f_i$ such that $j\in F_i$ (recall that $\bigcup \mathcal{F} = U)$. Note that $S'$ is indeed an interval set of $D$ as $v\in S$ and $|S'|\leq k+2$. Let $I = \{i\in\{1,\ldots,m\}\mid f_i\in S'\}$. Note that $\mathcal{F'} = \{F_i\mid i\in I\}$ satisfies $\bigcup \mathcal{F'} = U$ and $|\mathcal{F'}|\leq k$. Consequently, $\mathcal{I}$ is a YES-instance.
\end{proof}

\newcommand{\proboinpsparam}{\pname{Parameterized oriented $\opsc$-interval}\xspace}

\paraProblemDef{\proboinpsparam}{Oriented graph $D$ and positive integer $k$.}{$k$}{Is $\oinps(D)\leq k$?}{pro:oinpsparam}
    
\begin{cor}
\label{cor:oinps_split_w_hard}
 \nameref{pro:oinpsparam} is \W[2]-hard, even if $D$ has an underlying split graph. 
\end{cor}
\begin{proof}
    It suffices to observe in the proof of Proposition~\ref{prop:oinps_w-hard_split} that any path on 3 vertices that is used in the reduction is indeed a shortest path.
\end{proof}

\section{Algorithms for graphs with bounded clique-width}
\label{subsec:coucelle}

In this section, we define an interval or a hull set in the $\opc$ and $\opsc$ convexities of a given oriented graph $D$ in terms of monadic second-order logic. Thus, we obtain polynomial-time algorithms to compute these parameters for graphs with bounded clique-width. 

We start with a review about monadic second-order formulas on graphs. For more details, see~\cite{courcelle2}. We will use lowercase variables $x,y,z,\dots$ (resp. uppercase variables $X,Y,Z,\dots$) to denote vertices (resp. subsets of vertices) of graphs. The atomic formulas are $x=y$ and $x\in X$. The predicate $\arc(x,y)$ defines an arc from $x$ to $y$.  The set $MSO_1$ of monadic-second order is the set of formulas consisted from atomic formulas with Boolean connectives $\wedge, \vee, \Longrightarrow, \Longleftrightarrow$,  element quantifications $\exists x$ and $\forall x$ and set quantifications $\exists X$ and $\forall X$. A variable that is not linked to a quantifier is a free variable. We write $\varphi(x_1,\dots,x_m,Y_1,\dots,Y_q)$ to express a formula under the free variables $x_1,\dots,x_m,Y_1,\dots,Y_q$. A (di)graph $G$ models $\varphi$ if, when substituting the element variables by vertices of $G$, and set variables by vertex subsets of $G$, then $\varphi$ is always true. When it happens, we denote by $G\models \varphi (x_1,\dots,x_m,Y_1,\dots,Y_q)$. For a given $\varphi(x_1,\dots,x_m,Y_1,\dots,Y_q)$ formula in $MSO_1$, let $opt_\varphi\in\{\min,\max\}$ be the problem of determining the value $opt_\varphi \left( \sum_{i=1}^q |Z_i|\right)$ such that $G\models \varphi(x_1,\dots,x_m,Z_1,\dots,Z_q)$, among all possible choices of $Z_1, \ldots, Z_q\subseteq V(G)$.

The \emph{clique-width} is a graph complexity measure introduced by Coucelle and Olariu~\cite{COURCELLE200077}. Briefly, following how it is defined by \cite{MAMADOURUDINI}, a graph has \emph{clique-width} at most $k$ if it can be obtained from the empty graph by successively applying the operations (1) the disjoint union of two graphs, (2) add all edges between vertices labeled $i$ and vertices $j$, $i\neq j$, with $i,j\in \{1,\ldots,k\}$, (3) relabel the vertices labeled $i$ into $j$, with $i,j\in \{1,\ldots,k\}$ (4) creation of a graph with a single vertex labeled $i\in \{1,\dots,k\}$. The expression representing the sequence of operations to build a given graph is called the \emph{clique-width expression}. It is well-known that graphs of bounded tree-width, in particular trees, have bounded clique-width.


\begin{thm}(\!\!\cite{courcelle2})\label{mamadou}
    Let $k$ be a fixed constant. For every $\textrm{MSO}_1$~formula written as $\varphi(x_1,\ldots,x_m,Y_1,\ldots,Y_q)$, the problem $opt_{\varphi}$, for $opt_{\varphi}\in \{\text{min, max}\}$, can be solved in linear time in any graph of clique-width at most $k$, provided the clique-width expression is given. If the clique-width expression is not given, one can compute one using at most $2^{k+1}$ labels in cubic time. 
\end{thm}


We emphasize that Theorem~\ref{mamadou} holds for directed graphs when considering the clique-width of its underlying graph. The definition of a graph in terms of $\textrm{MSO}_1$ has as edges ordered pairs of vertices and the authors emphasize that we may consider arcs instead of edges. So, the term \emph{graph} in Theorem~\ref{mamadou} actually may be considered as an oriented graph and its \emph{width} is in fact the clique-width of its underlying graph. 

The following examples are formulas that express that there is a two-path linking $x$ and $y$; and also when that $x$ and $y$ are at distance two in a directed graph $D$ having a directed path $(x,z,y)$ in $D$:
\[2path(x,z,y):= \exists \arc(x,z)\wedge \exists \arc(z,y);\]
\[Ind2path(x,z,y):= \exists \arc(x,z)\wedge \exists \arc(z,y)\wedge\neg \exists \arc(x,y).\]

\begin{lem}
\label{lem:msol_interval}
We can describe an oriented interval set of a given oriented graph $D$ in the convexities $\opc$ and $\opsc$ with $\text{MSO}_1$.
\end{lem}
\begin{proof}
By definition, $S$ is an interval set in the $\opc$ if any vertex $z$ in $V(D)\setminus S$ lies in a directed two path between two vertices of $S$. Thus, we may represent it as:
\[IntervalSet_{\opc}(S):= \forall z (z\in S \vee \exists x \exists y (x\in S \wedge y\in S \wedge 2path(x,z,y))).\]

Similarly, one may deduce how to do it for the $\opsc$ convexity:

\[IntervalSet_{\opsc}(S):= \forall z (z\in S \vee \exists x \exists y (x\in S \wedge y\in S \wedge ind2path(x,z,y))).\]
\end{proof}

\begin{lem}
\label{lem:msol_hull}
We can describe an oriented hull set of a given oriented graph $D$ in the convexities $\opc$ and $\opsc$ with $\text{MSO}_1$.
\end{lem}

\begin{proof}
Initially, note that $ X \subseteq Y $ and $ X \subsetneq Y $ may be represented as $\forall x(x\in X\Longrightarrow x\in Y)$ and $\exists y(y\in Y\wedge y\notin X)\wedge X \subseteq Y $, respectively. Let us now present formulas to represent that a given set $S$ is convex. In the case of $\opc$-convexity we represent such property by 
\[Closed_{\opc}(X):= {\forall x,y}(x\in X \wedge y\in X\Rightarrow \neg \exists z(2path(x,z,y)\wedge z\notin X)).\]

Similarly, we present a formula to represent a convex set in the $\opsc$-convexity.
\[Closed_{\opsc}(X):={\forall x,y}(x\in X \wedge y\in X\Rightarrow \neg \exists z(
Ind2path(x,z,y)\wedge z\notin X)).\]

In conclusion, we present below the formulas that represent a hull set in $\opc$ and $\opsc$-convexities. Note that an equivalent definition of a hull set $S$ is that the unique convex set containing $S$ is $V(D)$, since its convex hull must be $V(D)$.

\[HullSet_{\opc}(S):= {\forall Z}(S \subseteq Z\wedge Z\subsetneq V(D)\Longrightarrow \neg Closed_{\opc}(Z) )\]

\[HullSet_{\opsc}(S):={\forall Z}(S \subseteq Z\wedge Z\subsetneq V(D)\Longrightarrow \neg Closed_{\opsc}(Z) ). \]
\end{proof}

\begin{prop}
Given an oriented graph $D$ on $n$ vertices such that the clique-width of its underlying graph is $k$, then there are algorithms that compute $\oinp(D)$,
$\oinps(D)$, $\ohnp(D)$ and $\ohnps(D)$ in cubic time.
\end{prop}

\begin{proof}
It is a direct consequence of Theorem~\ref{mamadou} and Lemmas~\ref{lem:msol_interval} and~\ref{lem:msol_hull}.
\end{proof}

\begin{omit}

\section{Further Research}
\label{sec:further}

    \julio{Find a good place to put this next result:}
    \begin{prop}
        For any graph $G$ and for any $\xc\in\{\gc,\pc,\psc\}$, there is an acyclic orientation $D$ of $G$ such that $\ohnx(D)\leq \hn_{\xc}(G)$.
    \end{prop}
    \begin{sketch}
        The idea is to iteratively build $D$ from an acyclic partial orientation $D'$ of $G$. We take an unoriented path $P$ of $G$ and orient it so that $P$ is a directed path in $D$, in such a way that $D$ is still acyclic while ensuring that $\ohnx(D)\leq \hn_{\xc}(G)$.
        
        Let $S$ be a minimum hull $\xc$-set of $G$. Note that $\hullx(S)$ is obtained by iteratively applying the interval function over $S$ until $\ifxc^k(S)=V(G)$. One iteration of the interval function adds, possibly, vertices that belong to distinct (shortest) $u,v$-paths (of length two). To build $D$, we will add such paths one by one.
        
        So one may consider that we can iteratively build $\hullx(S) = H_p$, for some $p\geq k$, as follows. At a step zero, we add $S$ to $H_0$. At step $i\geq 1$, we select \textbf{one} (shortest) $u,v$-path (of length two) $P_i$, for some $u,v\in H_{i-1}$ and add the inner vertices of $P_i$ to $H_i = H_{i-1}\cup V(P_i)$. 
        
        To build $D = D_p$, we construct partial acyclic orientations $D_i$ of $G[H_i]$, for $i\in\{1,\ldots,p\}$. First we take any acyclic orientation $D_0$ of $G[H_0]$. 
        At step $i\geq 1$, we check whether the endpoints $u,v\in V(P_i)$ of $P_i$ are comparable in $D_{i-1}$, i.e. if there is some directed $(u,v)$-path $Q_{\overrightarrow{uv}}$ or some directed $(v,u)$-path $Q_{\overrightarrow{vu}}$ in $D_{i-1}$. Note that we cannot have both $Q_{\overrightarrow{uv}}$ and $Q_{\overrightarrow{uv}}$, as  $D_{i-1}$ is acyclic. If $Q_{\overrightarrow{uv}}$ exists, then we orient $P_i$ from $u$ to $v$, otherwise we orient $P_i$ from $v$ to $u$. Note that $D_i$ is then acyclic. Moreover, note that since $P_i$ is a directed path in $D$, we have that the inner vertices of $P_i$ are added to $\ohullx(S)$ at step $i$, for each $i\in\{1,\ldots, p\}$. 
        Consequently, $\ohnx(D)\leq \hnx(G)$ as $S$ is a hull $\oxc$-set of $D$.
    \end{sketch}
    
    \julio{Check papers on $\hn^-$, $\hn^+$, $\gn^-$ and $\gn^+$. They probably contain next proposition w.r.t. the geodetic convexity.}
    \begin{prop}
        There exists a graph $G$ and an orientation $D$ of $G$ such that $\hnx(G)=\inx(G)>\oinx(D)=\ohnx(D)$, for $\xc\in\{\gc,\psc\}$.
    \end{prop}
    \begin{proof}
        It suffices to consider $G$ as a $C_3$ and $D$ as a directed $C_3$. Note that $\hnx(G)=\inx(G)=3>\oinx(D)=\ohnx(D)=2$.
    \end{proof}
    
\end{omit}

\section{Concluding remarks}\label{sec:conclusion}

    While we know exact values for the hull number of strong tournaments in the $\opc$, $\opsc$ and $\ogc$ convexities that imply polynomial-time algorithms to compute such values, the case of the interval number is more challenging. With respect to the $\opsc$ and $\ogc$ convexities, Proposition~\ref{prop:hull_tournament_sum} let us focus on the case of strong tournaments, as follows.
    
    \begin{prob}
        Can one compute  $\oing(T)$, when $T$ is a strong tournament, in polynomial time?
    \end{prob}
    
    For the case of $\opc$ convexity, it is wide open:
    
    \begin{prob}
    \label{prob:interval_p3_tournament}
        Can one compute $\oinp(T)$, when $T$ is a tournament, in polynomial time?
    \end{prob}
    However, one should notice that if a tournament is not strong, then any vertex that does not belong to the source nor the sink strong component lies in a directed $(u,v)$-path from any vertex $u$ in the source strong component to any $v$ in the sink strong component, meaning that focusing in the strong case may help.
    In Section~\ref{subsec:coucelle}, we used Theorem~\ref{mamadou} to compute $\oinp(D)$, $\oinps(D)$, $\ohnp(D)$ and $\ohnps(D)$ in cubic time, parameterized by the clique-width of the underlying graph of $D$. The geodetic case is harder. In the non-oriented case, Kante et al.~\cite{KMS19} showed that deciding whether $\hng(G)\leq k$ is $\W[1]$-hard parameterized by $k+tw(G)$. On the other hand, they provide an $\XP$ algorithm for graphs with bounded treewidth, which could probably be extended to the oriented case.
    
    \section*{Acknowledgements}
    
    Most of this work was done during a sabbatical year of J. Araujo and A. K. Maia at the AlgCo team (LIRMM, CNRS, Université de Montpellier, France) funded by CAPES 88887.466468/2019-00, Brazil.

\bibliographystyle{acm}
\bibliography{Oriented-Hull-Interval-ArXiv}
\newpage

\end{document}